\documentclass[11pt]{amsart}

\usepackage{amsmath,amssymb,amsthm,mathtools}
\usepackage[T1]{fontenc}
\usepackage{lmodern,microtype}
\usepackage{booktabs}
\usepackage{enumitem}
\usepackage[hidelinks]{hyperref}
\hypersetup{
  pdftitle={Transitive automorphism groups of maximal curves},
  pdfauthor={Saeed Tafazolian},
  pdfkeywords={maximal curve, Hermitian curve, Klein quartic, automorphism group}
}
\usepackage{xurl}
\usepackage[margin=1.05in]{geometry}

\newtheorem{theorem}{Theorem}[section]
\newtheorem{proposition}[theorem]{Proposition}
\newtheorem{lemma}[theorem]{Lemma}
\newtheorem{corollary}[theorem]{Corollary}
\theoremstyle{remark}

\newcommand{\F}{\mathbb F}
\newcommand{\Aut}{\operatorname{Aut}}
\newcommand{\Syl}{\operatorname{Syl}}
\newcommand{\PP}{\mathbb P}
\newcommand{\cH}{\mathcal H}

\title{Transitive automorphism groups of maximal curves}

\author{Saeed Tafazolian}
\address{Institute of Mathematics, Statistics and Scientific Computing,
University of Campinas, Brazil}
\email{saeed@unicamp.br}


\date{}

\begin{document}

\begin{abstract}
Let $q$ be an odd prime power and let $X$ be a maximal curve of genus at
least two over $\F_{q^2}$. We prove that the automorphism group of $X$
acts transitively on its rational points if and only if $X$ is the
Hermitian curve, or $q=5$ and $X$ is the unique $\F_{25}$-maximal curve
of genus $3$, namely the maximal Klein quartic. Together with the result
of Giulietti and Korchm\'aros in characteristic two, this gives the
classification for every prime power $q$. In particular, the Hermitian
curve is the only maximal curve of genus at least two whose automorphism
group is doubly transitive on its rational points.
\end{abstract}

\maketitle

\section{Introduction}

Let $q$ be a power of a prime $p$. A nonsingular projective,
geometrically irreducible curve $X$ of genus $g$ over $\F_{q^2}$ is
called \emph{maximal} if
\[
  \#X(\F_{q^2})=q^2+1+2gq.
\]
The genus of a maximal curve satisfies $g\le q(q-1)/2$, with equality
only for the Hermitian curve
\[
  \cH_{q+1}:\quad y^{q+1}=x^q+x
\]
by a theorem of R\"uck and Stichtenoth \cite{RS}. The group
$\Aut(\cH_{q+1})\cong PGU(3,q)$ acts doubly transitively on its
$q^3+1$ rational points \cite[Theorem~12.24]{HKT}. In this paper we
determine all maximal curves whose automorphism group is transitive on
the rational points.

We consider curves of genus at least two. For $g=0$ the group
$PGL(2,q^2)$ is transitive on the rational points, while for $g=1$
the translations by rational points give a transitive action. For a
maximal curve of genus at least two, every geometric automorphism is
defined over $\F_{q^2}$ (Lemma~\ref{lem:descent}). Thus $\Aut(X)$
acts on $X(\F_{q^2})$ and contains every subgroup of automorphisms
acting on this set. We may therefore consider the full group $\Aut(X)$.

The case of even $q$ was settled by Giulietti and Korchm\'aros
\cite[Theorem~9.1]{GK2010}: transitivity implies that the curve is
Hermitian. For $q=p\ge7$, Gunby, Smith and Yuan obtained the same
conclusion when the stabilizer of a rational point has order divisible
by $p$ \cite[Lemma~3.23]{GSY}. The Suzuki, Ree and
Giulietti--Korchm\'aros curves, which also have large automorphism
groups, have two orbits on their rational points
\cite[Section~6.1]{GKMSurvey}. We prove the following result in odd
characteristic.

\begin{theorem}\label{thm:main}
Let $q$ be a power of an odd prime $p$, and let $X/\F_{q^2}$ be a maximal
curve of genus $g\ge2$. Then $\Aut(X)$ acts transitively on $X(\F_{q^2})$
if and only if one of the following holds:
\begin{enumerate}[label=\textup{(\roman*)}]
\item $X$ is $\F_{q^2}$-isomorphic to $\cH_{q+1}$;
\item $q=5$ and $g=3$, in which case $X$ is the $\F_{25}$-maximal
Klein quartic.
\end{enumerate}
\end{theorem}

Combining Theorem~\ref{thm:main} with \cite[Theorem~9.1]{GK2010} we
obtain the following corollary.

\begin{corollary}\label{cor:all}
Let $q$ be any prime power and let $X/\F_{q^2}$ be a maximal curve of
genus $g\ge2$. Then $\Aut(X)$ acts transitively on $X(\F_{q^2})$ if and
only if either $X$ is Hermitian, or $q=5$ and $g=3$.
\end{corollary}

The exceptional action has order $168<56\cdot55$ and is not doubly
transitive. Consequently, the Hermitian curve is the only maximal curve
of genus at least two whose automorphism group acts doubly transitively
on its rational points.

The proof is divided into two cases. When $p\nmid|\Aut(X)|$, the Hurwitz
bound and the genus bounds for maximal curves leave a finite list of
tame triangle actions. All but the action of $PSL(2,7)$ on the Klein
quartic are excluded. When $p\mid|\Aut(X)|$, we use the $p$-rank zero
property to relate the rational points to Sylow $p$-subgroups; the
noncyclic and cyclic cases are treated separately. The new cases not
covered by \cite{GK2010,GSY} are the tame actions in odd
characteristic and the wild actions for $p=3,5$ and non-prime $q$.
No lower bound on $|\Aut(X)|$ is required.

Throughout the paper, $\Aut(X)$ denotes the geometric automorphism
group. Points and quotients are taken over $\overline{\F}_{q^2}$ unless
otherwise stated. For $G\le\Aut(X)$ and $P\in X$, write $G_P$ for the
stabilizer and $G_P^{(i)}$ for the lower ramification groups of the
local Galois extension induced by $X\to X/G$. In particular,
$G_P^{(1)}$ is the normal Sylow $p$-subgroup of $G_P$, and
$G_P/G_P^{(1)}$ is cyclic of order prime to $p$
\cite[Chapter~IV, \S\S1--2]{Serre}. We call an action \emph{tame}
if the order of the acting group is prime to $p$ and write
$n_p(G)=|\Syl_p(G)|$. A subgroup is \emph{TI} if distinct conjugates
intersect trivially. A finite group is \emph{almost simple} if it lies
between a nonabelian simple group and its automorphism group.

\section{Preliminaries}\label{sec:prelim}

\subsection{Maximal curves}

We recall the following genus bounds for maximal curves. Part (i)
is due to R\"uck and Stichtenoth \cite{RS}; the second-genus bound
in (ii) is due to Fuhrmann and Torres, in the form recalled in
\cite[Lemma~2.12 and Corollary~2.13]{KT}. The equality case (iii)
is \cite[Theorem~3.1]{FGT}, while (iv) is
\cite[Corollary~3.3]{KT}.

\begin{lemma}\label{lem:gaps}
Let $X/\F_{q^2}$ be maximal of genus $g$.
\begin{enumerate}[label=\textup{(\roman*)}]
\item $g\le q(q-1)/2$, with equality if and only if $X$ is Hermitian.
\item If $X$ is not Hermitian, then $g\le\lfloor(q-1)^2/4\rfloor$.
\item If $q$ is odd and $g=(q-1)^2/4$, then
$X\simeq_{\F_{q^2}}\{y^{(q+1)/2}=x^q+x\}$.
\item If $g<\lfloor(q-1)^2/4\rfloor$, then $g\le\lfloor(q^2-q+4)/6\rfloor$.
\end{enumerate}
\end{lemma}

The following result is due to Gunby, Smith and Yuan
\cite[Theorem~3.10]{GSY}.

\begin{lemma}\label{lem:descent}
Let $X$ be a maximal curve of genus $g\ge2$ over $\F_{q^2}$. Then every
geometric automorphism of $X$ is defined over $\F_{q^2}$.
\end{lemma}

\subsection{Tame actions and specialization}

\begin{lemma}\label{lem:threebranch}
Let $X$ be a curve of genus $g\ge2$ over an algebraically closed field of
characteristic $p>0$, and let $G\le\Aut(X)$ be tame.
\begin{enumerate}[label=\textup{(\roman*)}]
\item If $|G|>12(g-1)$, then $X/G\simeq\PP^1$ and $X\to X/G$ has exactly
three branch points. If their ramification indices are $a\le b\le c$, then
$p\nmid abc$ and
\begin{equation}\label{eq:defect}
  \frac{2(g-1)}{|G|}=1-\frac1a-\frac1b-\frac1c.
\end{equation}
\item If $|G|>48(g-1)$, the signature is $(2,3,7)$ and $|G|=84(g-1)$.
\end{enumerate}
\end{lemma}

\begin{proof}
Let $\bar g=g(X/G)$, and let $e_1,\ldots,e_r$ be the ramification
indices. Since the action is tame, Riemann--Hurwitz gives
\[
0<\frac{2(g-1)}{|G|}
=2\bar g-2+\sum_{i=1}^r\left(1-\frac1{e_i}\right).
\]
Suppose that $|G|>12(g-1)$. The right-hand side is less than $1/6$.
If $\bar g\ge2$, it is at least $2$. If $\bar g=1$, positivity forces
$r\ge1$, and the right-hand side is at least $1/2$. Thus $\bar g=0$.
For $r\le2$ the right-hand side is nonpositive, while for $r\ge5$
it is at least $1/2$. If $r=4$, it is either zero or at least $1/6$.
Consequently $r=3$, and (i) follows. Since each ramification index
divides $|G|$, we have $p\nmid abc$.

Now suppose that $|G|>48(g-1)$. By (i), the signature is $(a,b,c)$,
where $2\le a\le b\le c$, and
\[
0<1-\frac1a-\frac1b-\frac1c<\frac1{24}.
\]
If $a\ge3$, the smallest positive value of the middle expression
is $1/12$. Hence $a=2$. The case $b=2$ is impossible, and for
$b\ge4$ the smallest positive value is $1/20$. Thus $b=3$, and
$0<1/6-1/c<1/24$ gives $c=7$. Therefore $|G|=84(g-1)$.
\end{proof}

We use the following consequences of Grothendieck's specialization
theorem for prime-to-$p$ covers
\cite[Expos\'e~XIII, Lemme~2.11 and Corollaire~2.12]{SGA1}.

\begin{lemma}\label{lem:groth}
Let $k$ be an algebraically closed field of characteristic $p>0$,
and let $G$ be a finite group of automorphisms of a smooth projective
curve $X/k$, with $p\nmid |G|$. Suppose that $X/G\simeq\PP^1_k$ and
that the quotient map has $r$ branch points with ramification indices
$e_1,\ldots,e_r$. Then:
\begin{enumerate}[label=\textup{(\roman*)}]
\item There are $g_1,\ldots,g_r\in G$ such that
\[
G=\langle g_1,\ldots,g_r\rangle,\qquad
 g_1\cdots g_r=1,\qquad |g_i|=e_i,
\]
and $\langle g_i\rangle$ is an inertia group above the $i$-th branch
point.
\item There is a smooth projective complex curve $\widetilde X$ of
genus $g(X)$ with a faithful $G$-action and rational quotient, having
the same ramification indices.
\end{enumerate}
\end{lemma}

In particular, if $r=3$ and the indices are $(a,b,c)$, then
$G=\langle x,y\rangle$ with $|x|=a$, $|y|=b$ and $|xy|=c$.

\subsection{Group theory}

We use two theorems of Burnside \cite[Theorems~5.13 and~7.8]{Isaacs}: if
$P\in\Syl_p(G)$ and $P\le Z(N_G(P))$, then $G$ has a normal
$p$-complement; and every group of order $p^aq^b$ is solvable.

\begin{theorem}[Guralnick--Malmskog--Pries]\label{thm:GMP}
Let $p$ be odd, let $A$ be a finite group, and let $Q\in\Syl_p(A)$ be a
noncyclic TI subgroup with $N_A(Q)=Q\rtimes C$, where $C$ is cyclic of
order prime to $p$. If $N_A(Q)\ne A$, then $M:=Z(N_A(Q))=O_{p'}(A)$, the
group $A/M$ is almost simple, and $A$ is doubly transitive on
$\Syl_p(A)$.
\end{theorem}

The condition $N_A(Q)\ne A$ replaces $Q\ne A$ in the printed
statement of \cite[Theorem~3.16]{GMP}; the authors confirmed in
correspondence that this was a typographical error.

\section{The tame case}\label{sec:tame}

In this section $q=p^h$ is odd, $X/\F_{q^2}$ is a non-Hermitian maximal
curve of genus $g\ge2$, and $G\le\Aut(X)$ is a tame subgroup that is
transitive on $\Omega=X(\F_{q^2})$. Put
\[
  N=\#\Omega=q^2+1+2gq,\qquad s=|G_P|\ (P\in\Omega),\qquad |G|=sN .
\]
By the tame Hurwitz bound \cite[Theorem~11.56]{HKT},
\begin{equation}\label{eq:hurwitz}
  sN\le84(g-1).
\end{equation}

\subsection{\texorpdfstring{Reduction to $q=p\le19$}{Reduction to q = p <= 19}}

\begin{lemma}\label{lem:tamereduction}
Under the assumptions of this section, $q=p\le19$.
\end{lemma}

\begin{proof}
By Lemma~\ref{lem:gaps}(ii), $g\le(q-1)^2/4$; hence $q\ge5$. The function
$t\mapsto(q^2+1+2qt)/(t-1)=2q+(q+1)^2/(t-1)$ is decreasing for $t>1$, so
\begin{equation}\label{eq:ratio}
  \frac{N}{g-1}\ \ge\ 2q+4+\frac{16}{q-3}\ \ge\ 22 .
\end{equation}
By Lemma~\ref{lem:threebranch}(i), $X/G\simeq\PP^1$ with three branch
points. For $q\ge41$ the right side of \eqref{eq:ratio} exceeds $84$,
contradicting \eqref{eq:hurwitz}. For $23\le q\le37$ it exceeds $48$; then
$s=1$, and Lemma~\ref{lem:threebranch}(ii) gives signature $(2,3,7)$ and
$N=84(g-1)$. This is not tame for $q=27$, and for
$q\in\{23,25,29,31,37\}$ the resulting value $g=(q^2+85)/(84-2q)$ equals
$307/19$, $355/17$, $463/13$, $523/11$ or $727/5$, which is not an
integer.

The only remaining non-prime $q$ is $9$. Then $N/(g-1)\ge74/3$, so $s\le3$
by \eqref{eq:hurwitz}, and $s\in\{1,2\}$ since $3\nmid|G|$. If $s=2$, then
$|G|/(g-1)>48$ and the signature is $(2,3,7)$, which is not tame in
characteristic $3$. If $s=1$, the defect \eqref{eq:defect} is at most
$3/37$, so $1/a+1/b+1/c\ge34/37$ with $3\nmid abc$. If $a\ge4$ the sum is
at most $3/4$, and $a=3$ is not tame, so $a=2$. Then $b=2$ gives a
negative defect, $b=3$ is not tame, and $b\ge5$ gives a sum at most
$9/10<34/37$; hence $b=4$. Now $1/c\ge34/37-3/4=25/148$ gives $c\le5$, and
$c=4$ gives defect $0$, so $(a,b,c)=(2,4,5)$. Then $40(g-1)=N=82+18g$, so $g=61/11$, a
contradiction.
\end{proof}

\subsection{Numerical candidates}

\begin{lemma}\label{lem:candidates}
Under the assumptions of this section, $(p,g,N,s)$ and the signature of
$X\to X/G$ form one of the rows of Table~\ref{tab:candidates}.
\end{lemma}

\begin{table}[ht]
\centering
\caption{Numerical possibilities in the tame case.}
\label{tab:candidates}
\begin{tabular}{ccccc}
\toprule
$p$ & $g$ & $N$ & $s$ & signature\\
\midrule
5  & 3  & 56   & 3 & $(2,3,7)$\\
7  & 5  & 120  & 1 & $(2,3,10)$\\
11 & 9  & 320  & 1 & $(2,4,5)$\\
11 & 19 & 540  & 1 & $(2,3,10)$\\
11 & 25 & 672  & 3 & $(2,3,7)$\\
13 & 15 & 560  & 1 & $(2,4,5)$\\
19 & 41 & 1920 & 1 & $(2,3,8)$\\
\bottomrule
\end{tabular}
\end{table}

\begin{proof}
By Lemma~\ref{lem:tamereduction}, $q=p\le19$, and $p\ne3$ by
Lemma~\ref{lem:gaps}(ii). Put $R=|G|/(g-1)$. By \eqref{eq:ratio},
$R\ge22$, so the defect \eqref{eq:defect} equals $2/R\le1/11$. This leaves
the signatures $(2,3,c)$ with $7\le c\le13$, for which $R=12c/(c-6)$, and
$(2,4,5)$, $(2,4,6)$, $(3,3,4)$, with $R=40,24,24$. Combining
\eqref{eq:hurwitz} with \eqref{eq:ratio} gives $s\le3$ for $p\le11$,
$s\le2$ for $p\in\{13,17\}$ and $s=1$ for $p=19$. If $s>1$, then $\Omega$
is a short orbit, so $s\in\{a,b,c\}$. Finally, $R(g-1)=s(p^2+1+2pg)$
gives
\[
  g=\frac{R+s(p^2+1)}{R-2sp}.
\]
Substituting the possible values of $R$ and $s$ into this formula, with
$p\nmid abc$, $R>2sp$, and $g$ an integer satisfying
$2\le g\le(p-1)^2/4$, we obtain the seven cases in
Table~\ref{tab:candidates} and one additional case, namely
$p=17$, $g=55$, $s=1$, with signature $(2,4,5)$.
But for $p=17$ we have
\[
\left\lfloor\frac{p^2-p+4}{6}\right\rfloor=46<55<64=\frac{(p-1)^2}{4},
\]
contrary to Lemma~\ref{lem:gaps}(iv).
\end{proof}

\subsection{Exclusion of six rows}

\begin{lemma}\label{lem:easyrows}
The cases $(p,g)=(11,25)$ and $(19,41)$ are impossible.
\end{lemma}

\begin{proof}
Suppose first that $(p,g)=(11,25)$. By Lemma~\ref{lem:gaps}(iii),
$X\simeq y^6=x^{11}+x$, whose automorphism group has order
$6\cdot11\cdot(11^2-1)=7920$ \cite[Theorem~12.11]{HKT}. On the other
hand, transitivity implies that $N=672$ divides $|G|$, so $7$ divides
$|G|$. This is impossible, since $G\le\Aut(X)$ and $7\nmid7920$.

For $(p,g)=(19,41)$, we have $|G|=1920$. By
Lemma~\ref{lem:groth}(ii), $G$ acts on a compact Riemann surface of
genus $41$. But by Breuer's classification \cite[Chapter~5]{Breuer},
the order of such a group is at most $960$, a contradiction.
\end{proof}

\begin{lemma}\label{lem:genus19}
The case $(p,g)=(11,19)$ is impossible.
\end{lemma}

\begin{proof}
By Table~\ref{tab:candidates}, $|G|=540$ and the signature is $(2,3,10)$.
Hence Lemma~\ref{lem:groth} gives
\[
G=\langle x,y\rangle,\qquad |x|=2,\quad |y|=3,\quad |xy|=10.
\]
We first show that $G$ is not perfect. Suppose otherwise, and let $M$
be a maximal proper normal subgroup of $G$. Then $G/M$ is a
nonabelian simple group whose order divides $540$. The only possibility
is $G/M\simeq A_5$ \cite{SmallSimple}, and hence $|M|=9$. Since
$\gcd(9,10)=1$, the image of $xy$ in $G/M$ has order $10$. But $A_5$
has no element of order $10$, a contradiction.

Let $\bar x,\bar y$ denote the images of $x,y$ in $G/G'$. Since
\[
1=(\bar x\bar y)^{10}=\bar x^{10}\bar y^{10}=\bar y,
\]
this quotient is generated by $\bar x$ and has order at most $2$.
As $G$ is not perfect, we obtain
\[
G/G'\simeq C_2,\qquad |G'|=270.
\]

Put $H=\langle y,xyx\rangle$. Since $x^2=1$, conjugation by $x$
interchanges the generators of $H$. Thus $H$ is normalized by $x$
and also by $y$, so $H\triangleleft G$. Moreover, $y\in G'$ by the
preceding calculation, and hence $H\le G'$. On the other hand, $G/H$
is generated by the image of $x$, so $[G:H]\le2$. Since $[G:G']=2$,
we conclude that
\[
G'=H=\langle y,xyx\rangle.
\]

Let $T$ be a Sylow $2$-subgroup of $G'$. Since $|G'|=270$, we have
$|T|=2$, and $T\le Z(N_{G'}(T))$. By Burnside's normal
$2$-complement theorem, $G'$ has a normal subgroup $K$ of order $135$.
Thus $G'/K\simeq C_2$. Both $y$ and $xyx$ have order $3$, so their
images in $G'/K$ are trivial. This gives
$G'=\langle y,xyx\rangle\le K$, a contradiction.
\end{proof}

\begin{lemma}\label{lem:genus5}
The case $(p,g)=(7,5)$ is impossible.
\end{lemma}

\begin{proof}
By Table~\ref{tab:candidates}, $|G|=120$ and the signature is
$(2,3,10)$. By Lemma~\ref{lem:groth}, there are generators $x,y\in G$
such that $|x|=2$, $|y|=3$ and $|xy|=10$. Put $c=xy$. We show that
$z=c^5$ is central in $G$.

By Sylow's theorem, $n_5(G)$ is either $1$ or $6$. Suppose that
$n_5(G)=1$, and let $P$ be the normal Sylow $5$-subgroup. The image
of $c$ in $G/P$, of order $24$, has order dividing $2$. Thus $G/P$
is a quotient of $\Delta(2,3,2)\simeq S_3$, a contradiction. Therefore
$n_5(G)=6$.

Let $K$ be the kernel of the conjugation action of $G$ on these six
Sylow $5$-subgroups. Then $G/K$ embeds in $S_6$, and
$K\le N_G(P)$, which has order $20$. If $5$ divided $|K|$, the normal
subgroup $K$ would contain all six Sylow $5$-subgroups. Since they
intersect pairwise trivially, $K$ would have at least
$1+6(5-1)=25$ elements. Hence $|K|$ divides $4$.

As $S_6$ has no element of order $10$, and
$|\langle c\rangle\cap K|\le2$, the image of $c$ in $G/K$ has order
$5$. Thus $z=c^5\in K$. If $|K|=4$, then $G/K$ has order $30$ and is
a quotient of $\Delta(2,3,5)\simeq A_5$. This is impossible because
$A_5$ is simple. Hence $K=\langle z\rangle$ has order $2$ and is
normal in $G$. Therefore $z$ is central.

By Lemma~\ref{lem:groth}(i), the subgroup $\langle c\rangle$ is a
point stabilizer. Let $Q\in X$ satisfy $G_Q=\langle c\rangle$.
Since $z$ is central and fixes $Q$, it fixes all $120/10=12$ points
of $G\cdot Q$. For $Y=X/\langle z\rangle$, Riemann--Hurwitz gives
\[
8=2(2g(Y)-2)+r,\qquad r\ge12,
\]
where $r$ is the number of fixed points of $z$. Thus $g(Y)=0$.
By Lemma~\ref{lem:descent}, the involution $z$ and the quotient map
are defined over $\F_{49}$. Since $X$ has rational points, so does
$Y$, and hence $Y\simeq_{\F_{49}}\PP^1$. As the quotient map has
degree $2$, we obtain
\[
\#X(\F_{49})\le2(49+1)=100,
\]
contrary to maximality, which gives $\#X(\F_{49})=120$.
\end{proof}

For $(p,g)=(11,9)$ and $(13,15)$ the signature is $(2,4,5)$. We
first treat these two cases together.

\begin{lemma}\label{lem:245structure}
Suppose that $(p,g)$ is $(11,9)$ or $(13,15)$, so that $|G|=40(g-1)$
and the signature is $(2,4,5)$. Then:
\begin{enumerate}[label=\textup{(\roman*)}]
\item $G$ is solvable, $G/G'\simeq C_2$, $G'/G''\simeq C_5$, and
$U:=G''$ has order $4(g-1)$.
\item $T:=X/U\simeq\PP^1$, and $X\to T$ is branched at exactly five
points $t_1,\dots,t_5$, each with ramification index $2$. The group
$C:=G'/U\simeq C_5$ acts on $T$ and permutes $t_1,\dots,t_5$ cyclically.
\item Let $v_i\in A:=U/U'$ be the image of the nonidentity element of an
inertia group above $t_i$. Then $A\simeq\F_2^4$, the $v_i$ generate $A$,
and $v_1+\dots+v_5=0$ is their only linear relation. In particular,
$16$ divides $|U|$.
\end{enumerate}
\end{lemma}

\begin{proof}
By Lemma~\ref{lem:groth}, $G=\langle x,y\rangle$ with $|x|=2$, $|y|=4$
and $c:=xy$ of order $5$.

(i) The group $G$ is solvable: for $|G|=320$ by Burnside's $p^aq^b$
theorem, and for $|G|=560$ because no nonabelian simple group has order
dividing $560$ \cite{SmallSimple}. In $G/G'$ the image of $c$ has order
dividing both $5$ and $4$, so $\bar y=\bar x$, and $G/G'\simeq C_2$
because $G$ is solvable. Put $d=y^{-1}cy$; then $cd=y^2$. For
$H=\langle c,d\rangle$ we have
$y^{-1}cy=d$ and $y^{-1}dy=y^{-2}cy^2\in H$, since $y^2=cd\in H$.
Thus $y$ normalizes $H$, and $H\triangleleft G$ because
$G=\langle c,y\rangle$. Also $c\in G'$, so $H\le G'$. As $G/H$ is
generated by the image of $y$, whose square belongs to $H$, we have
$[G:H]\le2$. Therefore $H=G'$. In $H/H'$ the images of $c$ and $d$
have orders dividing $5$, while their product has order dividing $2$.
Hence their product is trivial, and $H/H'$ is cyclic of order dividing
$5$. As $H$ is nontrivial and solvable, $H/H'\simeq C_5$. Thus
$|U|=|G|/10$.

(ii) The group $D:=G/U$ has order $10$. The images of $x$ and $y$ in $D$
lie outside $G'/U\simeq C_5$, hence have order $2$, while $c$ meets $U$
trivially. Therefore the inertia groups $\langle x\rangle$,
$\langle y\rangle$, $\langle c\rangle$ meet $U$ in groups of order
$1,2,1$, and $X\to T$ ramifies exactly at the $|G|/4=10(g-1)$ points
above the branch point of index $4$, each with index $2$.
Riemann--Hurwitz gives
\[
  2g-2=4(g-1)\bigl(2g(T)-2\bigr)+10(g-1),
\]
so $g(T)=0$. Each ramified fibre has $|U|/2=2(g-1)$ points, so there are
$10(g-1)/(2(g-1))=5$ branch points. They form one $D$-orbit with
stabilizers of order $2$. Since $C$ has order $5$, it meets these
stabilizers trivially and acts freely and transitively on the five
branch points.

(iii) By Lemma~\ref{lem:groth}(i) applied to the $U$-cover $X\to T$,
$U$ is generated by involutions $u_1,\dots,u_5$ with $u_1\cdots u_5=1$,
where $u_i$ generates an inertia group above $t_i$. Hence the $v_i$
generate $A$, which is an elementary abelian $2$-group, nonzero because
$U$ is solvable, and $v_1+\dots+v_5=0$; thus $A\simeq\F_2^e$ with
$e\le4$. Conjugation by $G'$ induces an action of $C$ on $A$, since inner
automorphisms of $U$ act trivially on $A$. The inertia groups above $t_i$
are conjugate under $U$, so $v_i$ depends only on $t_i$, and conjugation
by $h\in G'$ maps $v(t)$ to $v(\bar h(t))$. Thus $C$ permutes
$v_1,\dots,v_5$ cyclically. This action is not trivial, since otherwise
all $v_i$ would equal some $v$ with $v=5v=0$, and then $A=0$.
As $C$ has prime order $5$, its nontrivial action is faithful. Hence
$C_5$ embeds in $GL(e,2)$, and $e=4$ because $5\nmid|GL(e,2)|$
for $e\le3$. Five vectors spanning $\F_2^4$ satisfy exactly one
nontrivial linear relation.
\end{proof}

\begin{lemma}\label{lem:genus15}
The case $(p,g)=(13,15)$ is impossible.
\end{lemma}

\begin{proof}
By Lemma~\ref{lem:245structure}(iii), the order of $U$ is divisible
by $16$. But $|U|=56$, a contradiction.
\end{proof}

\begin{lemma}\label{lem:genus9}
The case $(p,g)=(11,9)$ is impossible.
\end{lemma}

\begin{proof}
We keep the notation of Lemma~\ref{lem:245structure}. Here
$|U|=32$. As $v_1+\dots+v_5=0$ is the only relation among the $v_i$,
there is
a character $\chi\colon A\to C_2$ with $\chi(v_1)=0$ and $\chi(v_i)=1$
for $i=2,\dots,5$. A generator of $C$ acts on $T$ as $t\mapsto\zeta t$ with
$\zeta$ a primitive fifth root of unity, and after rescaling the branch
points are the fifth roots of unity $\{1,3,4,5,9\}\subset\F_{11}$, with
$v_1$ above $t=1$. Let $K$ be the kernel of $U\to A\xrightarrow{\chi}C_2$.
Then $E=X/K$ is a double cover of $T$ branched exactly at $3,4,5,9$, so
over $\overline{\F}_{11}$
\[
  E\simeq E_0\colon\ w^2=(t-3)(t-4)(t-5)(t-9)=t^4+t^3+t^2+t+1 .
\]
Let $c_{10}$ be the coefficient of $t^{10}$ in
$(t^4+t^3+t^2+t+1)^5$. The Hasse--Witt matrix of $E_0$
is $(c_{10})$. Using
$t^4+t^3+t^2+t+1=(1-t^5)/(1-t)$, we obtain
\[
  c_{10}=\binom{14}{4}-5\binom{9}{4}+10
        =381\equiv7\pmod{11}.
\]
Thus $E_0$, and hence $E$, is ordinary.
On the other hand, $X$ is superspecial since it is maximal over
$\F_{11^2}$ \cite[Proposition~2.5]{KNT}. As $E$ is a quotient of $X$,
it is also superspecial \cite[Corollary~2.8]{KNT}, a contradiction.
\end{proof}

\subsection{The Klein quartic}

\begin{proposition}\label{prop:klein}
Up to $\F_{25}$-isomorphism, there is a unique $\F_{25}$-maximal curve
of genus $3$. This is the Klein quartic $K$, whose automorphism group
is isomorphic to $PSL(2,7)$ and acts transitively on $K(\F_{25})$.
\end{proposition}

\begin{proof}
The uniqueness follows from \cite[Lemma~3.1(3)]{BMT}, while the
identification with the Klein quartic and its automorphism group are
given in \cite[\S3.1]{BTT}. Since $|\Aut(K)|=168=84(g(K)-1)$ and
$5\nmid168$, Lemma~\ref{lem:threebranch}(ii) shows that the action
has signature $(2,3,7)$. Thus its short orbits have lengths
$84$, $56$ and $24$, while the other orbits have length $168$.
By Lemma~\ref{lem:descent}, the set $K(\F_{25})$ is invariant under
$\Aut(K)$. Since it contains $56$ points, it must be the orbit of
length $56$.
\end{proof}

\begin{proposition}\label{prop:tamefinal}
Let $q$ be odd and let $X/\F_{q^2}$ be a non-Hermitian maximal curve of
genus at least two. If a tame subgroup of $\Aut(X)$ acts transitively on
$X(\F_{q^2})$, then $q=5$ and $X$ is the $\F_{25}$-maximal Klein quartic.
\end{proposition}

\begin{proof}
By Lemmas~\ref{lem:candidates}, \ref{lem:easyrows},
\ref{lem:genus19}, \ref{lem:genus5}, \ref{lem:genus15}
and \ref{lem:genus9}, the only possibility is
$(p,g)=(5,3)$. The result follows from
\cite[Lemma~3.1(3)]{BMT} and Proposition~\ref{prop:klein}.
\end{proof}

\section{The wild case}\label{sec:wild}

In this section $q=p^h$ is a power of an odd prime $p$, and
$X/\F_{q^2}$ is maximal of genus $g\ge2$. Put $G=\Aut(X)$ and
$\Omega=X(\F_{q^2})$. By Lemma~\ref{lem:descent}, $G$ acts on
$\Omega$. Assume this action is transitive and $p$ divides $|G|$.
Write $N=\#\Omega$. Since
$N\equiv1\pmod p$ and $|G|=N|G_P|$, the prime $p$ divides $|G_P|$. By
\cite[Lemma~1.1]{FGT}, the Frobenius of $J_X$ is $[-q]$, so $X$ has
$p$-rank zero.

\subsection{Sylow subgroups and rational points}

\begin{lemma}\label{lem:sylow}
Every nontrivial $p$-subgroup of $G$ fixes exactly one point of $X$, and
this point lies in $\Omega$. For $P\in\Omega$, the group
$S_P:=G_P^{(1)}$ is a Sylow $p$-subgroup of $G$, the map $P\mapsto S_P$ is
a $G$-equivariant bijection $\Omega\to\Syl_p(G)$, and
\[
  N_G(S_P)=G_P=S_P\rtimes C_m,\qquad p\nmid m .
\]
In particular $n_p(G)=N$, and the Sylow $p$-subgroups of $G$ are TI.
\end{lemma}

\begin{proof}
As $|\Omega|\equiv1\pmod p$, every nontrivial $p$-subgroup $U$ fixes
some point $P\in\Omega$. On a curve of $p$-rank zero an automorphism of
order $p$ has exactly one fixed point \cite[Lemma~11.129]{HKT}. If a
nonidentity element of $U$ fixed another point, a suitable power of
order $p$ would fix both points. Thus $P$ is the unique fixed point of
$U$, and every nonidentity element of $U$ fixes no other point.

Since $p\nmid N$, the $p$-parts of $|G|$ and $|G_P|$ agree.
The normal $p$-subgroup $S_P=G_P^{(1)}$ is therefore a Sylow
$p$-subgroup of $G$. Moreover, $G_P/S_P$ is cyclic of order $m$
prime to $p$. Choose $x\in G_P$ whose image generates this quotient.
Since $S_P$ is a $p$-group, $|x|=mp^b$ for some $b\ge0$. Then
$x^{p^b}$ has order $m$ and still generates $G_P/S_P$. Thus
\[
G_P=S_P\rtimes C_m.
\]

Conversely, let $S$ be a Sylow $p$-subgroup of $G$. By the first part,
$S$ fixes a unique point $P\in\Omega$. Since $S\le G_P$ and $S_P$
is the unique Sylow $p$-subgroup of $G_P$, we have $S=S_P$. Thus
$P\mapsto S_P$ is surjective, and it is injective by uniqueness of
the fixed point. The identity
\[
S_{\sigma(P)}=\sigma S_P\sigma^{-1}\qquad(\sigma\in G)
\]
shows that the bijection is $G$-equivariant.

Since $S_P\triangleleft G_P$, we have $G_P\le N_G(S_P)$. If
$\sigma\in N_G(S_P)$, then $\sigma(P)$ is also fixed by $S_P$,
whence $\sigma(P)=P$. Therefore $N_G(S_P)=G_P$.
Finally, if $S_P\cap S_Q$ contained a nonidentity element for $P\ne Q$,
that element would fix both points, contrary to the first part.
Hence the Sylow $p$-subgroups are TI.
\end{proof}

Since all Sylow $p$-subgroups of $G$ are conjugate,
they are either all cyclic or all noncyclic.
We treat these two cases separately.

\subsection{Noncyclic Sylow subgroups}

\begin{proposition}\label{prop:noncyclic}
If the Sylow $p$-subgroups of $G$ are noncyclic, then
$X\simeq_{\F_{q^2}}\cH_{q+1}$.
\end{proposition}

\begin{proof}
Let $S\in\Syl_p(G)$ with fixed point $P$. By Lemma~\ref{lem:sylow}, $S$
is TI, $N_G(S)=G_P=S\rtimes C_m$, and $[G:N_G(S)]=N>1$. By
Theorem~\ref{thm:GMP}, $M=Z(G_P)$ is a normal $p'$-subgroup of $G$, $G/M$
is almost simple, and $G$ is doubly transitive on $\Syl_p(G)$, hence on
$\Omega$. As $M\le G_P$ is normal in $G$, it fixes every point of $\Omega$: if $Q=\sigma(P)$, then $\sigma^{-1}M\sigma=M\le G_P$, so $M\le G_Q$. An
element of prime order $\ell\ne p$ has at most $2+2g/(\ell-1)\le2g+2$
fixed points by Riemann--Hurwitz, and $N>2g+2$; hence $M=1$ and $G$ is
almost simple.

Double transitivity gives $|G|\ge N(N-1)>4g^2q^2>8g^3$, using $N-1>2gq$
and $2g\le q(q-1)<q^2$. By Henn's classification \cite{Henn}, in the corrected form of
\cite[Theorem~11.127]{HKT}, in odd characteristic $X$ is then either
geometrically isomorphic to $y^2=x^r-x$ for a power $r$ of $p$, whose
hyperelliptic involution is central in $\Aut(X)$, or geometrically
Hermitian. An almost simple group has trivial center, so
$X_{\overline{\F}_p}\simeq\cH_{r+1}$ and $G\simeq PGU(3,r)$ for some power
$r$ of $p$.

The Hermitian curve $\cH_{r+1}$ is maximal over $\F_{r^2}$ and
$PGU(3,r)$ is transitive on its $r^3+1$ rational points, so
Lemma~\ref{lem:sylow} gives $n_p(PGU(3,r))=r^3+1$. Applied to $X$, it
gives $n_p(G)=N$. With $g=r(r-1)/2$ this reads
$r^3+1=q^2+1+r(r-1)q$, that is, $(q-r)(q+r^2)=0$. Hence $r=q$,
$g=q(q-1)/2$, and $X$ is Hermitian by Lemma~\ref{lem:gaps}(i).
\end{proof}

\subsection{Cyclic Sylow subgroups}

\begin{lemma}\label{lem:noncentral}
Assume that $X$ is not Hermitian and that $S\in\Syl_p(G)$ is cyclic.
Let $P$ be its unique fixed point and write
$G_P=S\rtimes C_m$. Then $C_m$ acts nontrivially on $S$ by
conjugation. If $p=3$, its conjugation image has order $2$.
\end{lemma}

\begin{proof}
Suppose that $C_m$ centralizes $S$. Since $S$ is cyclic,
$S\le Z(N_G(S))$. By Burnside's normal $p$-complement theorem,
there is a normal $p'$-subgroup $K$ of $G$ such that $G=KS$.
As $S$ fixes $P$, we have
\[
G\cdot P=K\cdot P=\Omega.
\]
Thus $K$ is tame and transitive on $\Omega$. By
Proposition~\ref{prop:tamefinal}, $X$ is the $\F_{25}$-maximal
Klein quartic. But its full automorphism group has order $168$,
which is prime to $5$, contrary to $p\mid|G|$.

Hence the conjugation action of $C_m$ on $S$ is nontrivial.
For $p=3$, its image has order prime to $3$ and divides
$|\Aut(C_{3^a})|=2\cdot3^{a-1}$, and therefore has order $2$.
\end{proof}

\begin{lemma}\label{lem:local}
Let a finite group act on a smooth curve over an algebraically closed
field of characteristic $p$, and let $P$ be a point with
\[
  I:=G_P=S\rtimes C_m,\qquad S=G_P^{(1)}\simeq C_{p^a},\quad n:=p^a>1,
  \quad p\nmid m .
\]
Let $c$ be the order of the action of $C_m$ on $S$ by conjugation, let
$j\ge1$ be the largest integer with $G_P^{(j)}=S$, and put
$\Delta=\sum_{i\ge1}(|G_P^{(i)}|-1)$. Then
\[
  c=\frac{m}{\gcd(m,j)},\qquad j\ge\frac mc,\qquad
  \Delta\ge\frac{m(n-1)}{c}.
\]
\end{lemma}

\begin{proof}
Let $\tau$ generate $C_m$. The action of $C_m$ on the cotangent
line at $P$ is faithful, since the kernel of the action of the
inertia group on that line is $I_1=S$
\cite[Chapter~IV, \S2]{Serre}. Let $\zeta$ be the eigenvalue of
$\tau$, which has order $m$. We may choose a local parameter $t$
with $\tau(t)=\zeta t$. Indeed, starting with any local parameter
$u$, one may take
\[
t=\frac1m\sum_{i=0}^{m-1}\zeta^{-i}\tau^i(u).
\]
Its linear term is that of $u$, since $\tau^i(u)\equiv\zeta^iu$
modulo the square of the maximal ideal, and it satisfies the
required eigenvalue relation.

Let $j$ be as in the statement. For $\sigma\in S$, write
\[
\sigma(t)=t+a_\sigma t^{j+1}+O(t^{j+2}).
\]
The map $\sigma\mapsto a_\sigma$ induces an injective homomorphism
\[
S/I_{j+1}\longrightarrow(k,+).
\]
Since $S$ is cyclic and the additive group of $k$ has exponent $p$,
this nontrivial quotient has order $p$. Hence $I_{j+1}=S^p$,
and the image of this homomorphism is a one-dimensional
$\F_p$-subspace of $k$.

The ramification-character relation
\cite[Chapter~IV, \S2, Proposition~9]{Serre} gives
\[
a_{\tau\sigma\tau^{-1}}=\zeta^j a_\sigma.
\]
Thus $\tau$ acts on $S/S^p\simeq C_p$ by multiplication by
$\zeta^j$, of order $m/\gcd(m,j)$. The kernel of the natural map
\[
\Aut(C_{p^a})\longrightarrow\Aut(C_p)
\]
is a $p$-group. Since the conjugation image of $C_m$ has
order prime to $p$, it has the same order on $S$ as on $S/S^p$.
Consequently,
\[
c=\frac{m}{\gcd(m,j)}.
\]
Finally, $j\ge\gcd(m,j)=m/c$, and the first $j$ positive
ramification groups all equal $S$, so
\[
\Delta\ge j(n-1)\ge\frac{m(n-1)}c.
\]
\end{proof}

\begin{proposition}\label{prop:cyclic}
If the Sylow $p$-subgroups of $G$ are cyclic, then $X$ is Hermitian.
\end{proposition}

\begin{proof}
Suppose that $X$ is not Hermitian. By Lemma~\ref{lem:gaps}(ii),
$q\ge5$. Let $S\in\Syl_p(G)$ and let $P\in\Omega$ be its unique
fixed point. Write
\[
|S|=n=p^a,\qquad G_P=S\rtimes C_m,\qquad s=nm,
\]
and let $c$ be the order of the conjugation image of $C_m$
in $\Aut(S)$. By Lemma~\ref{lem:noncentral}, this action is
nontrivial. Hence
\[
1<c\mid m,\qquad c\mid p-1.
\]
In particular, $m\ge2$.

We first determine the short orbits of $G$. Since $s=nm\ge5$,
and by \eqref{eq:ratio}, which follows from the non-Hermitian
genus bound,
\[
\frac{N}{g-1}\ge22,
\]
we have
\[
\frac{|G|}{g-1}=\frac{sN}{g-1}\ge110>84.
\]
By Lemma~\ref{lem:sylow}, a point fixed by an element of order
$p$ belongs to $\Omega$. Conversely, each point of $\Omega$ has
wild stabilizer, since $p\mid |G_P|$. Thus $\Omega$ is the only
wild short orbit.

Put
\[
\Delta=\sum_{i\ge1}(|G_P^{(i)}|-1).
\]
By Hilbert's different formula \cite[Chapter~IV, \S1]{Serre},
the different exponent at $P$ is
\[
d_P=s-1+\Delta.
\]
Since $\Delta\ge n-1\ge2$, we have $d_P>s$. If $X/G$ had
positive genus, Riemann--Hurwitz would give
\[
2g-2\ge\frac{|G|}{s}d_P>|G|>2g-2,
\]
a contradiction. Thus $X/G\simeq\PP^1$.

Let $b$ be the number of tame short orbits, with ramification
indices $e_1,\ldots,e_b$. Riemann--Hurwitz gives
\[
0<\frac{2g-2}{|G|}
=-2+\frac{d_P}{s}+\sum_{i=1}^{b}\left(1-\frac1{e_i}\right)
<\frac1{42}.
\]
As $d_P/s>1$ and each tame contribution is at least $1/2$,
we have $b\le2$. If $b=2$ and one of the indices is at least $3$,
the right-hand side exceeds $1/6$, which is impossible. Thus,
when $b=2$, both indices are $2$.

If $b=0$, Riemann--Hurwitz yields
\[
2g-2=N(d_P-2s),
\]
while if $b=2$ and $e_1=e_2=2$, it yields
\[
2g-2=N(\Delta-1).
\]
Both are impossible, since $0<2g-2<N$. We conclude that $b=1$.
Let $e$ denote the ramification index of this tame short orbit.
Using $|G|=sN$ and $d_P=s-1+\Delta$, we find
\[
\begin{aligned}
2g-2&=-2|G|+Nd_P+\frac{|G|}{e}(e-1)\\
&=-2sN+N(s-1+\Delta)+sN\left(1-\frac1e\right)\\
&=N\left(\Delta-1-\frac{s}{e}\right).
\end{aligned}
\]
Hence
\begin{equation}\label{eq:global}
\frac{2g-2}{N}=\Delta-1-\frac{s}{e}.
\end{equation}

As $N=q^2+1+2gq$, we have
\[
0<\frac{2g-2}{N}<\frac1q<1.
\]
Since $\Delta$ is an integer, we obtain $e\nmid s$ and
\[
\Delta=\left\lceil\frac{s}{e}\right\rceil+1.
\]
Writing $s=\alpha e+\rho$, where $1\le\rho\le e-1$, gives
\[
\frac{2g-2}{N}=\frac{e-\rho}{e}.
\]
As $e-\rho$ is a positive integer,
\[
e\ge\frac{N}{2g-2}
=q+\frac{(q+1)^2}{2g-2}>q.
\]
Since $s=nm$ and $e\nmid s$, the preceding expression for $\Delta$
gives
\[
\Delta=\left\lceil\frac{nm}{e}\right\rceil+1
<\frac{nm}{e}+2.
\]
On the other hand, Lemma~\ref{lem:local} gives
$\Delta\ge m(n-1)/c$. Combining these inequalities, we obtain
\begin{equation}\label{eq:master}
m\left(\frac{n-1}{c}-\frac ne\right)<2.
\end{equation}

We now distinguish the cases $q>p$ and $q=p$.

\emph{Case $q>p$.}
Here $q\ge p^2$. Since $c\le p-1$, $e>q$, and $n\ge p$, we have
\[
\frac{n-1}{c}-\frac ne
>\frac{n-1}{p-1}-\frac nq
\ge1-\frac pq\ge1-\frac1p.
\]
The second inequality follows because
$x\mapsto(x-1)/(p-1)-x/q$ is increasing for $x\ge p$.
By \eqref{eq:master},
\[
m<\frac{2p}{p-1}\le3.
\]
Since $1<c\mid m$, it follows that $m=c=2$.
Lemma~\ref{lem:local} and the upper bound for $\Delta$ now give
\[
n-1\le\Delta<\frac{2n}{e}+2,
\]
and consequently
\[
1-\frac3n<\frac2e<\frac2q.
\]
If $p\ge5$, the left-hand side is at least $2/5$, while the
right-hand side is at most $2/25$. If $p=3$ and $n\ge9$, the
corresponding bounds are $2/3$ and $2/9$. Both cases are
impossible.

Thus $p=n=3$. Then $s=6$ and $e>q\ge9$, so
\[
\Delta=\left\lceil\frac6e\right\rceil+1=2.
\]
Equation~\eqref{eq:global} gives
\[
\frac{2g-2}{N}=1-\frac6e>\frac13,
\]
contrary to $(2g-2)/N<1/q\le1/9$. Hence $q>p$ is impossible.

\emph{Case $q=p$.}
Since $X$ is not Hermitian and $g\ge2$, we have $p\ge5$.
From $c\le p-1$, $e>p$, and \eqref{eq:master}, we obtain
\[
m\left(\frac{n-1}{p-1}-\frac np\right)<2,
\]
or equivalently
\[
m(n-p)<2p(p-1).
\]
If $n\ge p^2$, then $m\ge2$ makes the left-hand side at least
$2p(p-1)$, a contradiction. Thus $n=p$.

Since $S$ has order $p$, its positive ramification groups all
equal $S$ up to the first jump $j$ and are trivial afterwards. Hence
\[
\Delta=j(p-1).
\]
By Lemma~\ref{lem:local},
\[
\Delta\ge\frac{m(p-1)}c\ge m.
\]
On the other hand, $e>p$ gives $pm/e<m$, and therefore
\[
\Delta=\left\lceil\frac{pm}{e}\right\rceil+1\le m+1.
\]
Consequently, $\Delta=m$ or $\Delta=m+1$.

Suppose first that $\Delta=m+1$. Then
\[
j(p-1)=m+1.
\]
Any common divisor of $m$ and $j$ divides $j(p-1)-m=1$.
Thus $\gcd(m,j)=1$, and Lemma~\ref{lem:local} gives $c=m$.
Since $c\mid p-1$, we have $m\le p-1$. It follows that $j=1$
and $m=p-2$. But $m\mid p-1$, so $p-2\mid p-1$, forcing
$p=3$, a contradiction.

We are left with $\Delta=m$. Then
\[
m=j(p-1),
\]
and Lemma~\ref{lem:local} gives $c=p-1$. In particular,
$p-1\mid m$. Write $e=p+k$, where $k\ge1$. From
\eqref{eq:global}, we obtain
\[
0<\frac{k(m-1)-p}{p+k}<\frac1p.
\]
Thus
\begin{equation}\label{eq:finalineq}
k(m-1)>p,\qquad
k\bigl(p(m-1)-1\bigr)<p(p+1).
\end{equation}

Suppose $p\ge7$. Since $m\ge p-1$, the second inequality in
\eqref{eq:finalineq} forces $k=1$. Indeed, if $k\ge2$ its left-hand
side is at least
\[
2\bigl(p(p-2)-1\bigr)>p(p+1).
\]
With $k=1$, the first inequality gives $m>p+1$. As $p-1\mid m$,
we must have $m\ge2(p-1)$. But the second inequality gives
$m\le p+2$, which is impossible for $p\ge7$.

Finally, let $p=5$. We have $4\mid m$, and \eqref{eq:finalineq}
becomes
\[
k(m-1)>5,\qquad k(5m-6)<30.
\]
These inequalities force $m=4$ and $k=2$, so $e=7$. Since
$\Delta=m=4$ and $s=pm=20$, equation~\eqref{eq:global} gives
\[
\frac{2g-2}{N}=3-\frac{20}{7}=\frac17.
\]
As $N=26+10g$, we obtain
\[
7(2g-2)=26+10g,
\]
and hence $g=10$. But $10=5(5-1)/2$, so $X$ is Hermitian by
Lemma~\ref{lem:gaps}(i), contrary to our assumption.
\end{proof}

\section{Proof of the main theorem}

\begin{proof}[Proof of Theorem~\ref{thm:main}]
The group $PGU(3,q)$ acts doubly transitively on the rational points of
the Hermitian curve. Proposition~\ref{prop:klein} gives transitivity
for the $\F_{25}$-maximal Klein quartic.

Conversely, suppose that $\Aut(X)$ is transitive on $X(\F_{q^2})$.
If $p$ divides $|\Aut(X)|$, then $X$ is Hermitian by
Propositions~\ref{prop:noncyclic} and~\ref{prop:cyclic}.
Otherwise $p\nmid|\Aut(X)|$, and $X$ is not Hermitian, since the
full group of the Hermitian curve has order divisible by $p$.
Proposition~\ref{prop:tamefinal} therefore gives the $\F_{25}$-maximal
Klein quartic.
\end{proof}

\section*{Acknowledgments}
This work was partially supported by the Brazilian funding agencies CNPq
(grant 302774/2025-4), FAPESP (grant 2024/00923-6) and FAEPEX (grant
3485/25).

\section*{Declaration of AI Use}
During the preparation of this work, the author used ChatGPT (OpenAI, GPT-6) for language editing, reference searches, computational cross-checks, and assistance in revising proofs. The author independently verified all mathematical arguments and takes full responsibility for the content of the manuscript.


\begin{thebibliography}{99}

\bibitem{BMT}
D.~Bartoli, M.~Montanucci and F.~Torres,
\emph{$\F_{p^2}$-maximal curves with many automorphisms are Galois-covered
by the Hermitian curve},
Adv. Geom. \textbf{21} (2021), 325--336.

\bibitem{BTT}
S.~Bootsma, S.~Tafazolian and J.~Top,
\emph{Maximal curves and Tate--Shafarevich results for quartic and sextic
twists},
Finite Fields Appl. \textbf{91} (2023), 102256.

\bibitem{Breuer}
T.~Breuer,
\emph{Characters and Automorphism Groups of Compact Riemann Surfaces},
London Math. Soc. Lecture Note Ser. \textbf{280},
Cambridge University Press, 2000.


\bibitem{FGT}
R.~Fuhrmann, A.~Garcia and F.~Torres,
\emph{On maximal curves},
J. Number Theory \textbf{67} (1997), 29--51.

\bibitem{GK2010}
M.~Giulietti and G.~Korchm\'aros,
\emph{Automorphism groups of algebraic curves with $p$-rank zero},
J. Lond. Math. Soc. (2) \textbf{81} (2010), 277--296;
doi:10.1112/jlms/jdp066.


\bibitem{GKMSurvey}
M.~Giulietti, G.~Korchm\'aros and M.~Montanucci,
\emph{Maximal curves over finite fields, past, present and future},
Panoramas et Synth\`eses \textbf{60} (2023), 37--66.

\bibitem{SGA1}
A.~Grothendieck,
\emph{Rev\^etements \'etales et groupe fondamental (SGA~1)},
Lecture Notes in Math. \textbf{224}, Springer, 1971.

\bibitem{GSY}
B.~Gunby, A.~Smith and A.~Yuan,
\emph{Irreducible canonical representations in positive characteristic},
Res. Number Theory \textbf{1} (2015), Art.~3;
doi:10.1007/s40993-015-0004-8.

\bibitem{GMP}
R.~Guralnick, B.~Malmskog and R.~Pries,
\emph{The automorphism groups of a family of maximal curves},
J. Algebra \textbf{361} (2012), 92--106.

\bibitem{SmallSimple}
M.~Hall, Jr.,
\emph{Simple groups of order less than one million},
J. Algebra \textbf{20} (1972), 98--102;
doi:10.1016/0021-8693(72)90090-7.

\bibitem{Henn}
H.-W.~Henn,
\emph{Funktionenk\"orper mit gro{\ss}er Automorphismengruppe},
J. Reine Angew. Math. \textbf{302} (1978), 96--115.

\bibitem{HKT}
J.~W.~P.~Hirschfeld, G.~Korchm\'aros and F.~Torres,
\emph{Algebraic Curves over a Finite Field},
Princeton University Press, 2008.

\bibitem{Isaacs}
I.~M.~Isaacs,
\emph{Finite Group Theory},
Grad. Stud. Math. \textbf{92}, Amer. Math. Soc., 2008.

\bibitem{KNT}
A.~Kazemifard, A.~R.~Naghipour and S.~Tafazolian,
\emph{A note on superspecial and maximal curves},
Bull. Iranian Math. Soc. \textbf{39} (2013), 405--413.

\bibitem{KT}
G.~Korchm\'aros and F.~Torres,
\emph{On the genus of a maximal curve},
Math. Ann. \textbf{323} (2002), 589--608.


\bibitem{RS}
H.-G.~R\"uck and H.~Stichtenoth,
\emph{A characterization of Hermitian function fields over finite fields},
J. Reine Angew. Math. \textbf{457} (1994), 185--188.

\bibitem{Serre}
J.-P.~Serre,
\emph{Local Fields},
Grad. Texts in Math. \textbf{67}, Springer, 1979.

\end{thebibliography}
\end{document}